\documentclass[12pt, reqno]{amsart}
\usepackage{graphicx}
\usepackage{amscd}
\usepackage{amsmath}
\usepackage{amsfonts}
\usepackage{amssymb}
\usepackage{cite}
\usepackage{float}
\usepackage{caption}
\usepackage{subcaption}
\usepackage{leftidx}
\usepackage{mathtools}
\usepackage{enumitem}
\usepackage{pgf,tikz}
\usepackage{chngcntr}
\usepackage{mathtools}
\mathtoolsset{showonlyrefs}

\def\R{\mathbb{R}}

\def\P{\mathcal{P}}

\def\ds{\mathrm{d}s}
\def\B{\mathcal{B}}

\def\ds{\displaystyle}

\def\K{\mathcal{K}}

\newtheorem{theorem}{Theorem}
\newtheorem{corollary}{Corollary}
\newtheorem{lemma}{Lemma}

\newtheorem{definition}{Definition}

\newtheorem{example}{Example}

\numberwithin{equation}{section}
\numberwithin{theorem}{section}
\numberwithin{definition}{section}
\numberwithin{lemma}{section}
\numberwithin{corollary}{section}
\numberwithin{proposition}{section}
\numberwithin{remark}{section}
\numberwithin{example}{section}
\numberwithin{table}{section}

\newcounter{romnum}

\begin{document}

\title{EXISTENCE OF POSITIVE SOLUTIONS OF A HAMMERSTEIN INTEGRAL EQUATION USING THE LAYERED COMPRESSION-EXPANSION FIXED POINT THEOREM}


\author[S. Dhar]{Sougata Dhar}
	\address{Department of Mathematics, University of Connecticut, Storrs, CT 06269 USA}
	\email{sougata.dhar@uconn.edu}
	
\author[J. W. Lyons]{Jeffrey W. Lyons}
	\address{The Citadel, Charleston, SC 29409, USA}
	\email{jlyons3@citadel.edu}
	
\author[J. T. Neugebauer]{Jeffrey T. Neugebauer}
	\address{Department of Mathematics and Statistics, Eastern Kentucky University, Richmond, KY 40475, USA}
	\email{Jeffrey.Neugebauer@eku.edu}		



\address{$^1$ Department of Mathematics, University of Connecticut, Storrs, CT 06269, USA, \\ 
$^2$ The Citadel, Charleston, SC 29409, USA,\\
$^3$ Department of Mathematics and Statistics, Eastern Kentucky University, Richmond, KY 40475, USA}

\begin{abstract} 
In this paper, we show the existence of a positive solution of a Hammerstein integral equation under certain conditions on the kernel.  We apply the recent Layered Compression-Expansion Fixed Point Theorem. Finally, we provide corollaries to help in application of the main results and present an example.
\end{abstract}
\maketitle

\setcounter{equation}{0}
\section{Introduction}

Let $T_1,T_2\in\R$ with $T_1<T_2$. Consider the Hammerstein integral equation
\begin{equation}\label{1}
x(t)=\int^{T_2}_{T_1}G(t,\tau)f(x(\tau))\;d\tau, \quad t\in[T_1,T_2],
\end{equation}
where $f\in C([0,\infty),[0,\infty))$ and $f=f_\uparrow +f_\downarrow$ is the sum of monotonic increasing and decreasing functions, respectively. 

Under certain conditions, we show the existence of positive solutions to \eqref{1}. Our approach is to use the Layered Expansion-Compression Fixed Point Theorem \cite{AAH18} with an appropriately defined cone. A couple of additional assumptions yield the existence of positive symmetric solutions to \eqref{1}. Here, we impose symmetry on the problem and investigate the half interval. Once we obtain the existence of a positive solution on the half interval, we extend to the full interval using symmetry.

Avery type fixed point theorems (for some recent examples, see \cite{AAH16, AAH18-1, AAH19, AHL15}) have been used extensively to show the existence of positive solutions and positive symmetric solutions of boundary value problems involving differential equations, difference equations, time scale equations, and fractional differential equations. When Avery type fixed point theorems are applied, they are applied to boundary value problems whose corresponding Green's functions have properties similar to the properties that will be imposed on the kernel $G$; see, for example, \cite{AAHL12,AAH10,AAHL11,
 AAHL12-2,AAH11,AHO08,AAH16,AAH18-1,AGL15,DLN19, HLLN10,LNS12,LN14,LN15, LN15-2,NS17}. Recently, Eloe and Neugebauer \cite{EN19}, motivated by Webb \cite{W17}, applied the Avery type fixed point theorem found in \cite{AAH09} to a Hammerstein integral equation. They found the properties required on the kernel for positive solutions or positive symmetric solutions were properties possessed by the Green's function corresponding to, for example, a fractional boundary value problem, a right focal boundary value problem, a conjugate boundary value problem, and a Lidstone boundary value problem. In this paper, we show that under the same assumptions on the kernel $G$, the Layered Expansion-Compression Fixed Point Theorem can be applied to show the existence of positive and positive symmetric solutions of the Hammerstein integral equation. More examples of recent work on Hammerstein integral equations, can be found in \cite{BIK17,CIP17,IM17,LR19,PR19}.

The remainder of this paper is organized as follows. In Section 2, we present the definitions of convex and concave functionals and the Layered Compression-Expansion Fixed Point Theorem. Our main result for the existence of positive solutions for \eqref{1} is found in Section 3. Finally, we show the existence of positive symmetric solutions for \eqref{1} in Section 4. The paper concludes with an example involving a fourth order boundary value problem with Lidstone boundary conditions.

\section{Preliminaries}

\begin{definition}
Let $E$ be a real Banach space. A nonempty closed convex set $\mathcal{P} \subset E$ is called a cone if for all $x \in \mathcal{P}$ and $\lambda \geq0$, $\lambda x \in \mathcal{P},$ and if $x, -x \in \mathcal{P}$ then $x = 0.$
\end{definition}

\begin{definition}
A map $\alpha$ is said to be a nonnegative continuous concave functional on a cone $\mathcal{P}$ of a real Banach space $E$ if $\alpha : \mathcal{P} \to [0, \infty)$ is continuous and if
$$\alpha(tx + (1-t)y) \geq t\alpha(x) + (1-t)\alpha(y)$$ for all $x, y \in \mathcal{P}$ and $t \in [0,1].$ Similarly, we say the map $\beta$ is a nonnegative continuous convex functional on a cone $\mathcal{P}$ of a real Banach space $E$ if $\beta: \mathcal{P} \to [0, \infty)$ is continuous and if $$\beta(tx + (1-t)y) \leq t\beta(x) + (1-t)\beta(y)$$ for all $x, y \in \mathcal{P}$ and $t \in [0,1].$
\end{definition}

Let $\P$ be a cone, let $u$ and $v$ be real numbers, $\phi$ be a continuous concave functional and $\xi$ a continuous convex functional. Define $$\P(\phi,u,\xi,v)=\{x\in\P:\phi(x)<u \mbox{ and }\xi(x)>v\}.$$

The following is the fixed point theorem that will be utilized.

\begin{lemma}[Layered Compression-Expansion FPT \cite{AAH18}]\label{fpt} Suppose $\P$ is a cone in a real Banach space $E$, $\alpha$ and $\psi$ are nonnegative continuous concave functionals on $\P$, $\beta$ and $\theta$ are nonnegative continuous convex functionals on $\P$, and $R,S,T$ are completely continuous operators on $\P$ with $T=R+S$. If there exist nonnegative real numbers $a,b,c,d$ and $(r_0,s_0)\in\P(\beta,b,\alpha,a)\times\P(\theta,c,\psi,d)$ such that
\begin{enumerate}
\item[(A0)] $\P(\beta,b,\alpha,a)\times\P(\theta,c,\psi,d)$ is bounded;
\item[(A1)] if $r\in\partial\P(\beta,b,\alpha,a)$ with $\alpha(r)=a$ and $s\in\overline{\P(\theta,c,\psi,d)}$, then $\alpha(R(r+s))>a$;
\item[(A2)] if $r\in\partial\P(\beta,b,\alpha,a)$ with $\beta(r)=b$ and $s\in\overline{\P(\theta,c,\psi,d)}$, then $\beta(R(r+s))<b$;
\item[(A3)] if $s\in\partial\P(\theta,c,\psi,d)$ with $\theta(s)=c$ and $r\in\overline{\P(\beta,b,\alpha,a)}$, then $\theta(S(r+s))<c$; and
\item[(A4)] if $s\in\partial\P(\theta,c,\psi,d)$ with $\psi(s)=d$ and $r\in\overline{\P(\beta,b,\alpha,a)}$, then $\psi(S(r+s))>d$;
\end{enumerate}
then there exists an $(r^*,s^*)\in \P(\beta,b,\alpha,a)\times\P(\theta,c,\psi,d)$ such that $x^*=r^*+s^*$ is a fixed point for $T$.
\end{lemma}

\section{Positive Solutions of the Hammerstein Equation}

The following assumptions are made on $G$.
\begin{enumerate}
\item[(H1)] $G\in C([T_1,T_2]\times[T_1,T_2],[0,\infty))$ and $G(t,s)\not\equiv0$.
\item[(H2)] For each $\tau$, if $t_1,t_2\in[T_1,T_2]$ with $t_1\le t_2$, then $G(t_1,\tau)\le G(t_2,\tau)$.
\item[(H3)] There exists a $k>0$ such that for any $y,w\in[T_1,T_2]$ with $y\le w$,
\begin{equation*}
	(y-T_1)^kG(w,\tau)\le (w-T_1)^kG(y,\tau).
\end{equation*}
\end{enumerate}

Notice assumption (H3) gives that
\begin{equation}\label{gprop}
(y-T_1)^k\int^{T_2}_{T_1}G(w,\tau)\;d \tau\le (w-T_1)^k\int^{T_2}_{T_1}G(y,\tau)\;d \tau.
\end{equation}

Let $\B=C([T_1,T_2],\R)$ be the Banach Space composed of continuous functions defined from $[T_1,T_2]$ into $\R$
with the norm
\begin{equation*}
\|x\|=\max_{t\in[T_1,T_2]}|x(t)|.
\end{equation*}

Define the operator $T:\B\rightarrow\B$ by
$$Tx(t)=\int^{T_2}_{T_1}G(t,\tau)f(x(\tau))\;d \tau, \quad t\in[T_1,T_2].$$
Clearly, fixed points of $T$ are solutions of \eqref{1}.
Denote $R:\B\to \B$ by
$$
Rx(t)=\int^{T_2}_{T_1} G(t,\tau)f_\uparrow(x(\tau))\;d\tau,\quad t\in[T_1,T_2],
$$ and $S:\B\to \B$ by
$$
Sx(t)=\int^{T_2}_{T_1} G(t,\tau)f_\downarrow(x(\tau))\;d\tau,\quad t\in[T_1,T_2].
$$

Define the cone $\P\subset \B$ by
\begin{align*}
\P &=\left\{x\in \mathcal{B} : x \mbox{ is nonnegative, nondecreasing, and} \right.\\
&\phantom{==}\left.(y-T_1)^kx(w)\le (w-T_1)^kx(y)\mbox{ for all } y,w\in[T_1,T_2]\mbox{ with }y\le w \right\}.
\end{align*}

The following proof can be found in \cite{EN19}.
\begin{lemma}\label{conemap} Assume (H1)--(H3). Then the operators $T,R,S:\P\rightarrow\P$ and are completely continuous.
\end{lemma}

Define $\bar{T}=\dfrac{T_1+T_2}{2}$. Then $\dfrac{T_1+\bar{T}}{2}=\dfrac{3T_1+T_2}{4}$. 

For $x\in\P$, define the nonnegative continuous convex functionals $\beta$ and $\theta$ on $\P$ by
$$
\beta(x)=\max_{t\in[T_1,T_2]}x(t)=x(T_2),
$$ and
$$
\theta(x)=\max_{t\in\left[T_1,\frac{3T_1+T_2}{4}\right]}x(t)=x\left(\frac{3T_1+T_2}{4}\right),
$$
and the nonnegative continuous concave functionals $\alpha$ and $\psi$ on $\P$ by
$$
\alpha(x)=\min_{t\in\left[\frac{3T_1+T_2}{4},T_2\right]}x(t)=x\left(\frac{3T_1+T_2}{4}\right),
$$ and
$$
\psi(x)=x(T_2).
$$

Now, we show the existence of solutions to the Hammerstein equation.

\begin{theorem}\label{main} Assume (H1)--(H3). Let $a,b,c,d$ be nonnegative real numbers with $b>a$. If $f_\downarrow,f_\uparrow:[0,\infty)\to[0,\infty)$ are continuous with

\begin{enumerate}
    \item[(1)] $\ds f_{\uparrow}\left(a+4^{-k}d\right)>a\left(\int_{\frac{3T_1+T_2}{4}}^{T_2}G\left(\frac{3T_1+T_2}{4},\tau\right)d\tau\right)^{-1}$; 
    
    \item[(2)] $\ds f_{\uparrow}\left(b+4^kc\right)<b\left(\int_{T_1}^{T_2}G(T_2,\tau)\;d\tau\right)^{-1}$;
    
    \item[(3)] $\ds f_{\downarrow}(0)<c\left(\int_{T_1}^{T_2}G\left(\frac{3T_1+T_2}{4},\tau\right)d\tau\right)^{-1}$; and
    
    \item[(4)] $\ds f_{\downarrow}(b+d)>d\left(\int_{T_1}^{T_2}G(T_2,\tau)\;d\tau\right)^{-1}$;
\end{enumerate}
then there exists a fixed point $x^*\in\mathcal{P}$ for $T$ which is a solution for the Hammerstein equation \eqref{1}. By construction, $x^*$ is positive.
\end{theorem}

\begin{proof} 
Our goal is to apply Lemma \ref{fpt} to \eqref{1}.
In particular, we want to show the existence of a fixed point of $T$, which is a solution of \eqref{1}. Let $(r,s)\in\P(\beta,b,\alpha,a)\times\P(\theta,c,\psi,d)$ and $a$, $b$, $c$, and $d$ be nonnegative real numbers. First note that $a<b$, since
$$a<\alpha(r)=r\left(\frac{3T_1+T_2}{4}\right)<r(T_2)=\beta(r)<b.$$

Now, we show that $(A0)$ holds.  Since $(r,s)\in\P(\beta,b,\alpha,a)\times\P(\theta,c,\psi,d)$, we have
$$r(t)\le r(T_2)=\beta(r)< b,$$
and so $\|r\|< b$. Next, $\theta(s)=s\left(\frac{3T_1+T_2}{4}\right)< c$, so since $s\in\P$,
$$\left(\frac{3T_1+T_2}{4}-T_1\right)^ks(T_2)\leq\left(T_2-T_1\right)^k s\left(\frac{3T_1+T_2}{4}\right).$$
Isolate $s(T_2)$, and we have
$$\ds s(T_2)\leq \left(T_2-T_1\right)^k /\left(\frac{3T_1+T_2}{4}-T_1\right)^k s\left(\frac{3T_1+T_2}{4}\right)=4^ks\left(\frac{3T_1+T_2}{4}\right).$$
Therefore,
$$s(t)\le s(T_2)\le4^ks\left(\frac{3T_1+T_2}{4}\right)< 4^k c.$$
Thus, $\|s\|$ is bounded, and so $\P(\beta,b,\alpha,a)\times\P(\theta,c,\psi,d)$ is bounded.

Next, we show $(A1)$ is true. Let $r\in\partial\P(\beta,b,\alpha,a)$ with $\alpha(r)=a$ and $s\in\overline{\P(\theta,c,\psi,d)}$. Since $s\in\P$,
$$s\left(\frac{3T_1+T_2}{4}\right)\ge 4^{-k}s(T_2).$$
Then, since $r$ and $s$ are nondecreasing, for all $\tau\in\left[ \frac{3T_1+T_2}{4},T_2\right]$, it must be the case that
\begin{align*}
r(\tau)+s(\tau )&\ge r\left(\frac{3T_1+T_2}{4}\right)+s\left(\frac{3T_1+T_2}{4}\right)\\&\ge r\left(\frac{3T_1+T_2}{4}\right)+4^{-k}s(T_2)\\&\ge a + 4^{-k}d.
\end{align*}
Since $f_\uparrow$ is increasing, for $\tau\ge \dfrac{3T_1+T_2}{4}$, $f_\uparrow\left(r(\tau)+s(\tau)\right)\ge f_\uparrow\left(a +4^{-k} d\right)$.
So, $(1)$ gives that
\begin{align*}
\alpha\left(R(r+s)\right)&=\int^{T_2}_{T_1} G\left(\frac{3T_1+T_2}{4},\tau\right)f_\uparrow\left(r(\tau)+s(\tau)\right)d\tau\\
&\ge \int^{T_2}_{\frac{3T_1+T_2}{4}} G\left(\frac{3T_1+T_2}{4},\tau\right)f_\uparrow\left(r(\tau)+s(\tau)\right)d\tau\\
&\ge f_\uparrow\left(a + 4^{-k}d\right)\int^{T_2}_{\frac{3T_1+T_2}{4}} G\left(\frac{3T_1+T_2}{4},\tau\right)d\tau\\
&>a.
\end{align*}
So, $(A1)$ holds.

Here we show $(A2)$ holds. Let $r\in\partial\P(\beta,b,\alpha,a)$ with $\beta(r)=b$ and $s\in\overline{\P(\theta,c,\psi,d)}$. Again, note in this case that $s(T_2)\le 4^k c$. Then, since $r$ and $s$ are nondecreasing, for all $\tau\in[T_1,T_2]$, $$r(\tau)+s(\tau)\le r(T_2)+s(T_2)\le b+4^k c.$$ The fact that $f_\uparrow$ is increasing gives that $f_\uparrow\left(r(\tau)+s(\tau)\right)\le f_\uparrow\left(b+4^k c\right)$. By $(2)$,
\begin{align*}
\beta\left(R(r+s)\right)&=\int^{T_2}_{T_1} G(T_2,\tau)f_\uparrow\left(r(\tau)+s(\tau)\right)d\tau\\
&\le f_\uparrow\left(b+4^k c\right)\int^{T_2}_{T_1} G(T_2,\tau)\;d\tau\\
&<b.
\end{align*}
So, $(A2)$ holds.

To show $(A3)$ holds, we start by letting $s\in\partial\P(\theta,c,\psi,d)$ with $\theta(s)=c$ and $r\in\overline{\P(\beta,b,\alpha,a)}$. Notice since $f_\downarrow$ is decreasing, $f_\downarrow(0)\ge f_\downarrow(r(\tau)+s(\tau))$ for all $\tau\in[T_1,T_2]$. So, $(3)$ gives that
\begin{align*}
\theta\left(S(r+s)\right)&=\int^{T_2}_{T_1} G\left(\frac{3T_1+T_2}{4},\tau\right)f_\downarrow(r(\tau)+s(\tau))\;d \tau\\
&\le f_\downarrow(0)\int^{T_2}_{T_1} G\left(\frac{3T_1+T_2}{4},\tau\right)d\tau\\
&<c.
\end{align*}

Finally, we show $(A4)$ holds. Let $s\in\partial\P(\theta,c,\psi,d)$ with $\psi(s)=d$ and $r\in\overline{\P(\beta,b,\alpha,a)}$. Since $r$ and $s$ are nondecreasing, for all
$\tau\in[T_1,T_2],$ we have
$$
r(\tau)+s(\tau)\le r\left(T_2\right)+s\left(T_2\right)\le b+d.
$$

Because $f_\downarrow$ is decreasing,
$f_\downarrow\left(r(\tau)+s(\tau)\right)\ge f_\downarrow\left(b+d\right)$
for $\tau\in[T_1,T_2]$.
By $(4)$,
\begin{align*}
\psi(S(r+s))&=\int^{T_2}_{T_1}G(T_2,\tau)f_\downarrow\left(r(\tau)+s(\tau)\right)d\tau\\
&\ge f_\downarrow\left(b+d\right)\int^{T_2}_{T_1} G(T_2,\tau)\;d\tau\\
&>d.
\end{align*}
Therefore, the conditions of Lemma \ref{fpt} are satisfied, which gives the existence of a fixed point $x^*$ of $T$. This fixed point $x^*$ is a solution of the Hammerstein equation \eqref{1}.

\end{proof}

We present the following two corollaries of Theorem \ref{main} that provide information to assist in choosing the nonnegative real numbers $a$, $b$, $c$, and $d$ in application problems.

\begin{corollary}
    Suppose the conditions of Theorem \ref{main}. Then 
    \begin{enumerate}
        \item [(1)] $d<c\left(\dfrac{\ds\int_{T_1}^{T_2}G(T_2,\tau)\;d\tau}{\ds\int_{T_1}^{T_2}G\left(\dfrac{3T_1+T_2}{4},\tau\right)d\tau}\right)$; and
        
        \item [(2)] if $d<4^{2k}c$, then $a<b\left(\dfrac{\ds\int_{\frac{3T_1+T_2}{4}}^{T_2}G\left(\frac{3T_1+T_2}{4},\tau\right)d\tau}{\ds\int_{T_1}^{T_2}G(T_2,\tau)\;d\tau}\right).$
    \end{enumerate}
\end{corollary}
The proof of this corollary is direct from the decreasing and increasing properties of $f_\downarrow$ and $f_\uparrow$, respectfully.

\begin{corollary}
    Assume (H1)--(H3). Let $b$ and $c$ be positive real numbers. If $f_\downarrow,f_\uparrow:[0,\infty)\to[0,\infty)$ are continuous with
    \begin{enumerate}
    \item[(C1)] $\ds f_{\uparrow}(0)>0$;
    \item[(C2)] $\ds f_{\uparrow}\left(b+4^kc\right)<b\left(\int_{T_1}^{T_2}G(T_2,\tau)\;d\tau\right)^{-1}$;
    
    \item[(C3)] $\ds f_{\downarrow}(0)<c\left(\int_{T_1}^{T_2}G\left(\frac{3T_1+T_2}{4},\tau\right)d\tau\right)^{-1}$; and
    
    \item[(C4)] $\ds f_{\downarrow}(b)>0$;
    \end{enumerate}
    then there exists as least one positive solution of \eqref{1}.
\end{corollary}
The proof of this corollary is immediately apparent by choosing $a=d=0$ in Theorem \ref{main}.

\section{Positive Symmetric Solutions of the Hammerstein Equation}

In this section, we seek to show the existence of at least one positive symmetric solution to the Hammerstein equation. Here we mean symmetric in the sense that $x(t)=x(T_2-t+T_1)$ for $t\in[T_1,T_2]$. Our strategy is to divide the interval $[T_1,T_2]$ in half and employ the method of the previous section. If we prove the existence of a positive solution on the half interval, then the symmetry condition will extend the solution to the entire interval.

Recall $\bar{T}=\dfrac{T_1+T_2}{2}$. Define $\bar{\bar{T}}=\dfrac{T_1+\bar{T}}{2}$. Then $\dfrac{T_1+\bar{\bar{T}}}{2}=\dfrac{7T_1+T_2}{8}$. 

Define the cone
\begin{align*}
\K&=\left\{x\in\B:x(T_2-t+T_1)=x(t),\; x\ge0, \mbox{ nondecreasing for } x\in[T_1,\bar{T}]\right.\\
&\phantom{==} \mbox{ and}\left.(y-T_1)^kx(w)\le (w-T_1)^kx(y)\mbox{ for all } y,w\in[T_1,\bar{T}]\mbox{ with }y\le w \right\}.
\end{align*}

We need the following additional assumptions.
\begin{enumerate}
\item[(H4)] Let $t_1,t_2\in\left[T_1,\bar{T}\right]$ with $t_1\le t_2$.
\begin{enumerate}
\item[(i)] If $t_2\le \tau\le T_2-t_2+T_1$, then $G(t_1,\tau)\le G(t_2,\tau)$.
\item[(ii)] If $\tau\le t_2$, then $G(t_1,\tau)+G(T_2-t_1+T_1,\tau)\le G(t_2,\tau)+G(T_2-t_2+T_1,\tau)$.
\end{enumerate}
\item[(H5)] For all $t,\tau\in[T_1,T_2]$, $$G(T_2-t+T_1,T_2-\tau+T_1)=G(t,\tau).$$
\end{enumerate}

\begin{lemma} Assume (H1), (H3)--(H5). The operators $T,R,S:\K\rightarrow\K$ are completely continuous.
\end{lemma}
This proof can also be found in \cite{EN19} and is therefore omitted.

For $x\in\K$, define the nonnegative continuous convex functionals $\beta$ and $\theta$ on $\K$ by
$$
\beta(x)=\max_{t\in[T_1,\bar{T}]}x(t)=x(\bar{T}),
$$ and
$$
\theta(x)=\max_{t\in\left[T_1,\frac{7T_1+T_2}{8}\right]}x(t)=x\left(\frac{7T_1+T_2}{8}\right),
$$
and the nonnegative continuous concave functionals $\alpha$ and $\psi$ on $\K$ by
$$
\alpha(x)=\min_{t\in\left[\frac{7T_1+T_2}{8},\bar{T}\right]}x(t)=x\left(\frac{7T_1+T_2}{8}\right),
$$ and
$$
\psi(x)=x(\bar{T}).
$$

Finally, we present the main result of this section to show the existence of a positive solution on the half interval $[T_1,\bar{T}]$. By symmetry, we extend this solution to the entire interval $[T_1,T_2]$. We omit the proof as it is similar to the proof presented in the previous section.

\begin{theorem}\label{mainsym} Assume (H1), (H3)--(H5). Let $a$, $b$, $c$, $d$ be nonnegative real numbers with $b>a$. If $f_\downarrow,f_\uparrow:[0,\infty)\to[0,\infty)$ are continuous with
\begin{enumerate}
    \item[(1)] $\ds f_{\uparrow}\left(a+4^{-k}d\right)>a\left(\int_{\frac{7T_1+T_2}{8}}^{\bar{T}}G\left(\frac{7T_1+T_2}{8},\tau\right)d\tau\right)^{-1}$; 
    
    \item[(2)] $\ds f_{\uparrow}\left(b+4^kc\right)<b\left(\int_{T_1}^{T_2}G(\bar{T},\tau)\;d\tau\right)^{-1}$;
    
    \item[(3)] $\ds f_{\downarrow}(0)<c\left(\int_{T_1}^{T_2}G\left(\frac{7T_1+T_2}{8},\tau\right)d\tau\right)^{-1}$; and
    
    \item[(4)] $\ds f_{\downarrow}(b+d)>d\left(\int_{T_1}^{\bar{T}}G(\bar{T},\tau)\;d\tau\right)^{-1}$;
\end{enumerate}
then there exists a fixed point $x^*\in\K$ for $T$. Due to the symmetry of elements in $\K$, $x^*$ is a positive symmetric solution for the Hammerstein equation \eqref{1}.
\end{theorem}

We present the following two corollaries of Theorem \ref{mainsym} that provide information to assist in choosing the nonnegative real numbers $a$, $b$, $c$, and $d$ in application problems.

\begin{corollary}
    Suppose the conditions of Theorem \ref{mainsym}. Then 
    \begin{enumerate}
        \item [(1)] $d<c\left(\dfrac{\ds\int_{T_1}^{\bar{T}}G(\bar{T},\tau)\;d\tau}{\ds\int_{T_1}^{T_2}G\left(\dfrac{7T_1+T_2}{8},\tau\right)d\tau}\right)$; and
        
        \item [(2)] if $d<4^{2k}c$, then $a<b\left(\dfrac{\ds\int_{\frac{7T_1+T_2}{8}}^{\bar{T}}G\left(\frac{7T_1+T_2}{8},\tau\right)d\tau}{\ds\int_{T_1}^{T_2}G(\bar{T},\tau)\;d\tau}\right).$
    \end{enumerate}
\end{corollary}
The proof of this corollary is direct from the decreasing and increasing properties of $f_\downarrow$ and $f_\uparrow$, respectfully.

\begin{corollary}
    Assume (H1), (H3)--(H5). Let $b$ and $c$ be positive real numbers. If $f_\downarrow,f_\uparrow:[0,\infty)\to[0,\infty)$ are continuous with 
    \begin{enumerate}
    \item[(C1)] $\ds f_{\uparrow}(0)>0$;
    \item[(C2)] $\ds f_{\uparrow}\left(b+4^kc\right)<b\left(\int_{T_1}^{T_2}G(\bar{T},\tau)\;d\tau\right)^{-1}$;
    
    \item[(C3)] $\ds f_{\downarrow}(0)<c\left(\int_{T_1}^{T_2}G\left(\frac{7T_1+T_2}{8},\tau\right)d\tau\right)^{-1}$; and
    
    \item[(C4)] $\ds f_{\downarrow}(b)>0$,
    \end{enumerate}
    then there exists at least one positive symmetric solution to \eqref{1}.
\end{corollary}
The proof of this corollary is again apparent by choosing $a=d=0$ in Theorem \ref{mainsym}.

\begin{example} We remark that any integral equation or boundary value problem referred to in \cite{EN19} can be shown to have a positive solution by applying either Theorem \ref{main} or Theorem \ref{mainsym}, since the properties of the Green's functions associated with those boundary value problems are the properties of the kernel given either in Section 3 or in Section 4. 

We consider a specific example 
\begin{equation}\label{lid1}
x^{(4)} =f(x):=f_\uparrow(x)+f_\downarrow(x), \quad t\in(0,1),
\end{equation}
satisfying the Lidstone boundary conditions 
\begin{equation}\label{lid2}
x(0)=x''(0)=x(1)=x''(1)=0.
\end{equation}
The Green's function is given by $$G(t,\tau)=\left\{\begin{array}{ll} \dfrac{\tau^3t-3\tau^2t+\tau t^3+2\tau t-t^3}{6},&0\le t\le \tau\le 1,\\\\ \dfrac{\tau^3 t-\tau^3+\tau t^3-3\tau t^2+2\tau t}{6}, & 0\le \tau\le t\le 1;\end{array}\right.$$
i.e., $x$ is a solution of \eqref{lid1}, \eqref{lid2} if and only if $x$ satisfies \eqref{1} with the given kernel $G$.
Here $G$ satisfies properties (H1) and (H3)-(H5) with $k=1$ (see \cite{EN19}).

Now $$\int^{1}_0 G\left(\frac{1}{2},\tau\right)d\tau=2\int^{1/2}_0 G\left(\frac{1}{2},\tau\right)d\tau =2\int^{1/2}_0 \left(\frac{1}{16}\tau -\frac{1}{12}\tau^3\right) d\tau=\frac{5}{384},$$
and 
\begin{align*}
\int^{1}_0 G\left(\frac{1}{8},\tau\right)d\tau&=
2\int^{1/2}_0 G\left(\frac{1}{8},\tau\right)d\tau\\ &=2\int^{1/8}_0\left(-\frac{7}{48}\tau^3+\frac{35}{1024}\tau\right)d\tau\\
&\phantom{=}+ 2\int^{1/2}_{1/8}\left(\frac{1}{48}\tau^3-\frac{1}{16}\tau^2+\frac{43}{1024}\tau-\frac{1}{3072}\right)d\tau\\
&=\frac{277}{49152}.
\end{align*}
By choosing $b=1$ and $c=\dfrac{1}{4}$, we have that by Theorem \ref{mainsym}, \eqref{lid1}, \eqref{lid2} is guaranteed to have a symmetric solution if
$$f_\uparrow(0)>0,$$
$$f_\uparrow(2)<\frac{384}{5},$$
$$f_\downarrow(0)<\frac{12288}{277},$$
and $$f_\downarrow(1)>0.$$ It is easy to find a multitude of increasing and decreasing functions that satisfy these conditions.

\end{example}

\footnotesize


\end{document}